\documentclass[11pt,a4paper]{article}

\usepackage[T1]{fontenc}
\usepackage[utf8]{inputenc}
\usepackage[english]{babel}
\usepackage{lmodern}
\usepackage{microtype}
\usepackage{amsmath,amssymb,amsthm,mathtools,mathrsfs}
\usepackage{booktabs}
\usepackage{enumitem}
\usepackage{xcolor}
\usepackage[a4paper,hmargin=2.5cm,vmargin=2.3cm]{geometry}
\usepackage[colorlinks=true,linkcolor=blue!55!black,
  citecolor=blue!55!black,urlcolor=blue!55!black]{hyperref}
\hypersetup{
  pdftitle={The maximum number of Pareto eigenvalues of a real matrix of order four is 23},
  pdfauthor={Samir Adly},
  pdfsubject={Pareto eigenvalues and Pareto capacity},
  pdfkeywords={Pareto eigenvalue, eigenvalue complementarity problem,
    Pareto capacity, fixed point index}
}
\newtheorem{theorem}{Theorem}[section]
\newtheorem{proposition}[theorem]{Proposition}
\newtheorem{lemma}[theorem]{Lemma}
\newtheorem{corollary}[theorem]{Corollary}
\theoremstyle{definition}
\newtheorem{definition}[theorem]{Definition}
\newtheorem{remark}[theorem]{Remark}

\newcommand{\R}{\mathbb R}
\newcommand{\Ind}{\mathcal I}
\newcommand{\Supp}{\mathscr S}
\newcommand{\Id}{\operatorname{Id}}
\newcommand{\one}{\mathbf 1}
\newcommand{\indic}{\mathbf I}
\newcommand{\supp}{\operatorname{supp}}
\newcommand{\Fix}{\operatorname{Fix}}
\newcommand{\ind}{\operatorname{ind}}
\DeclareMathOperator{\sgn}{sgn}

\makeatletter
\renewenvironment{abstract}
  {%
    \par\smallskip
    \small
    \noindent\textbf{Abstract. }\ignorespaces
  }
  {%
    \par\medskip
  }
\makeatother
\title{The maximum number of Pareto eigenvalues of a real matrix of order four is 23}
\author{Samir Adly\thanks{XLIM, Universit\'e de Limoges,
123 avenue Albert Thomas, 87060 Limoges Cedex, France.
Email: \href{mailto:samir.adly@unilim.fr}{samir.adly@unilim.fr}.}}
\date{}

\begin{document}

\maketitle

\begin{abstract}
For a given real matrix $A\in\R^{n\times n}$, a Pareto eigenvalue of $A$ is a real number $\lambda\in\R$ for which there exists a nonzero vector $x\in\R^n\setminus\{0\}$ such that
$$
0\leq x\perp Ax-\lambda x\geq0.
$$
We prove that every matrix $A\in\R^{4\times4}$ has at most $23$ distinct Pareto eigenvalues. We first prove the result for matrices of order $4$ satisfying two conditions: every real eigenvalue of a principal submatrix is simple, and two different principal submatrices have no real eigenvalue in common. For these matrices, a fixed point argument shows that the number of Pareto eigenvalues is odd. Previous known results show that the Pareto capacity of order $4$ is between $23$ and $26$. Thus only $25$ remains to exclude. We exclude this case by using the support profiles in order $3$ and identities involving eigenvectors of principal submatrices. A perturbation and fixed point index argument then extends the bound to all real matrices of order $4$. An exact symbolic computation certifies that an explicit matrix of order $4$ has exactly $23$ distinct regular Pareto eigenvalues.
\end{abstract}



\section{Introduction}
\label{sec:introduction}

The Pareto eigenvalue problem is the eigenvalue complementarity problem associated with the positive orthant $K=\R_+^n$. Let $A\in\R^{n\times n}$. A real number $\lambda$ is called a Pareto eigenvalue of $A$ if there exists a vector $x\in\R^n\setminus\{0\}$ such that
$$
x\geq0,
\qquad
Ax-\lambda x\geq0,
\qquad
\langle x,Ax-\lambda x\rangle=0.
$$
The set of all Pareto eigenvalues of $A$ is denoted by $\Pi(A)$. In this paper, we study the maximum possible number of distinct Pareto eigenvalues of a real matrix. The Pareto capacity of order $n\geq1$ is defined by
$$
c_n:=\max_{A\in\R^{n\times n}}|\Pi(A)|.
$$
A matrix $A$ is called full-capacity if $|\Pi(A)|=c_n$.

Seeger's support criterion \cite[Theorem~4.1]{Seeger1999} gives a finite description of the Pareto spectrum. Let $J$ be a nonempty subset of $\{1,\ldots,n\}$. We say that $J$ produces $\lambda$ if there exists a vector $u\in\R^{|J|}$ such that
\begin{equation}
\label{ext-equation}
A_Ju=\lambda u,
\qquad
u>0,
\qquad
A_{J^{\mathrm c},J}u\geq0.
\end{equation}
This production is called regular if $\lambda$ is an algebraically simple eigenvalue of $A_J$ and all the exterior inequalities in \eqref{ext-equation} are strict. The same Pareto eigenvalue may be produced by different supports. Therefore, counting the eigenvalues of the principal submatrices is not sufficient to determine $|\Pi(A)|$.

Before this work, the exact capacities were known only up to order three:
$$
c_1=1,
\qquad
c_2=3,
\qquad
c_3=9.
$$
The equality $c_3=9$ was proved by Baillon and Seeger \cite{BaillonSeeger2021}. The support profiles of the full-capacity matrices of order three were classified in \cite{Adly2026}. In order four, the known bounds were
$$
23\leq c_4\leq26.
$$
The lower bound is given by an explicit matrix communicated by Victor Qi and reported by Adly and Seeger
\cite[Table~2]{AdlySeeger2011}. The upper bound follows from the paragraph after Corollary~1 in Baillon and Seeger \cite[p.~8]{BaillonSeeger2021}, where that corollary is combined with Seeger's support count \cite[Lemma~5.1]{Seeger1999}.

In this paper, we close the gap between these two bounds. Our main result is stated below.

\begin{theorem}
\label{thm:main}
Every real matrix of order $4$ has at most $23$ distinct Pareto eigenvalues, and this bound is attained. Equivalently,
$$
c_4=23.
$$
\end{theorem}

Baillon and Seeger introduced the regular Pareto capacity $c_n^{\mathrm{reg}}$ by maximizing the number of distinct Pareto eigenvalues that admit a regular production \cite{BaillonSeeger2020}. Kielstra proved in his doctoral thesis that
$$
c_4=c_4^{\mathrm{reg}}
$$
\cite[Corollary~6.1.7]{Kielstra2023}. His proof uses an exhaustive search over sign configurations, followed by perturbation arguments. We give an independent proof of the reduction needed here. We first prove that $c_4^{\mathrm{reg}}=23$ and then extend this result to all real matrices of order $4$.

The proof of the upper bound is divided into three steps. First, assume that a matrix of order $4$ has more than $23$ distinct Pareto eigenvalues admitting regular productions. A sufficiently small perturbation continues these productions and can be chosen principally simple. The positive-part equation of Adly and Rammal \cite{AdlyRammal2013} gives a fixed point formulation on the standard simplex. The Lefschetz-Hopf formula then shows that the number of Pareto eigenvalues of the perturbed matrix is odd. This matrix has at least $24$ Pareto eigenvalues, while the general upper bound is $26$. It must therefore have exactly $25$ Pareto eigenvalues.

Second, we use the order-three support profiles and several determinant identities. These restrictions exclude the ten candidate detailed profiles with $25$ values, up to a simultaneous permutation of the indices. This proves
$$
c_4^{\mathrm{reg}}=23.
$$
Third, let $A$ be a real matrix of order four with at least $24$ Pareto eigenvalues. A bound on the number of principal root occurrences, together with simultaneous regularization and continuation of the fixed point index, produces a nearby principally simple matrix having at least as many Pareto eigenvalues. This contradicts the regular bound and completes the proof of Theorem~\ref{thm:main}.

The paper is organized as follows. Section~\ref{sec:supports} introduces support production, regularity, and principally simple matrices. It also presents the matrix with $23$ regular Pareto eigenvalues. Section~\ref{sec:parity} gives the fixed point formulation and proves the parity result. Section~\ref{sec:local-identities} establishes the restrictions used in order four. Section~\ref{sec:regular-capacity} determines the regular capacity. Section~\ref{sec:nonregular} extends the regular bound to all real matrices of order $4$. Section~\ref{sec:proof-main} completes the proof of Theorem~\ref{thm:main}. The proof of simultaneous regularization is given in Appendix~\ref{app:simultaneous-regularization}.

\section{Supports, regularity, and principal simplicity}
\label{sec:supports}

For $n\geq1$, set
$$
\Ind_n:=\{1,\ldots,n\},
\qquad
\Supp_n:=\{J\subseteq\Ind_n:J\neq\varnothing\}.
$$
Vector inequalities are understood componentwise. For
$J\in\Supp_n$, set $J^{\mathrm c}:=\Ind_n\setminus J$. We write
$\R^J$ for the coordinate space indexed by $J$, $A_J$ for the
principal submatrix indexed by $J$, and $A_{J^{\mathrm c},J}$ for the
block with rows in $J^{\mathrm c}$ and columns in $J$. The symbol
$\one_J$ denotes the all-ones vector in $\R^J$, while an unsubscripted
$\one$ has the dimension required by the context. For a statement
$\mathcal E$, $\indic_{\{\mathcal E\}}$ denotes its indicator. The
ordinary spectrum of a square matrix $M$ is denoted by
$\operatorname{spec}(M)$.

\subsection{Support productions and assigned profiles}

The first definition records the support of a Pareto eigenvector.

\begin{definition}
A support $J\in\Supp_n$ produces $\lambda\in\R$ if there is
a vector $u\in\R^{|J|}$ such that
$$
A_Ju=\lambda u,
\qquad
u>0,
\qquad
A_{J^{\mathrm c},J}u\geq0.
$$
The vector $A_{J^{\mathrm c},J}u$ is the exterior slack.
The set of values produced by $J$ is denoted by $\Pi_J(A)$, and
$$
p_J(A):=|\Pi_J(A)|.
$$
\end{definition}

The complementarity conditions are exactly equivalent to this
description.

\begin{proposition}
\label{prop:support-criterion}
For every $A\in\R^{n\times n}$,
$$
\Pi(A)
=
\bigcup_{J\in\Supp_n}\Pi_J(A).
$$
\end{proposition}

\begin{proof}
Let $(\lambda,x)$ satisfy the Pareto conditions and put
$J=\supp x$. Complementarity gives
$$
A_Jx_J=\lambda x_J,
\qquad
x_J>0,
\qquad
A_{J^{\mathrm c},J}x_J\geq0.
$$
Conversely, extend a producing vector by zero outside $J$. The
resulting vector satisfies the Pareto conditions.
\end{proof}

Several supports may produce the same value, so we shall choose one
support for each distinct value.

\begin{definition}
An assigned support map for $A$ is a map
$$
\tau:\Pi(A)\longrightarrow\Supp_n
$$
such that $\lambda\in\Pi_{\tau(\lambda)}(A)$. For such a map, put
$$
m_J(A,\tau)
:=
|\{\lambda\in\Pi(A):\tau(\lambda)=J\}|,
\qquad
S_k(A,\tau)
:=
\sum_{\substack{J\in\Supp_n\\ |J|=k}}m_J(A,\tau).
$$
\end{definition}

The sets
$$
\{\lambda\in\Pi(A):\tau(\lambda)=J\},
\qquad
J\in\Supp_n,
$$
form a partition of $\Pi(A)$, whereas the sets $\Pi_J(A)$ may overlap.
Consequently,
$$
|\Pi(A)|=\sum_{k=1}^nS_k(A,\tau),
\qquad
0\leq m_J(A,\tau)\leq p_J(A)\leq|J|.
$$
Kielstra's Pareto profile records the intrinsic numbers $p_J(A)$
\cite[Definition~3.6.6]{Kielstra2023}. The assigned profile records
$m_J(A,\tau)$ and counts every distinct Pareto value exactly once.
We call $(m_J(A,\tau))_{J\in\Supp_n}$ the detailed assigned profile
and $(S_1(A,\tau),\ldots,S_n(A,\tau))$ its aggregate profile.
When $A$ and $\tau$ are fixed, we omit them from $m_J(A,\tau)$ and
$S_k(A,\tau)$. Braces and commas are also omitted in support
subscripts; thus $m_{ij}=m_{\{i,j\}}(A,\tau)$.

\subsection{Regularity and principally simple matrices}

We now define regular productions, which persist under small
perturbations.

\begin{definition}
A support $J$ regularly produces $\lambda$ if $\lambda$ is an
algebraically simple eigenvalue of $A_J$ and
$$
A_{J^{\mathrm c},J}u>0
$$
for a positive eigenvector $u$. For $J=\Ind_n$, the exterior condition is
empty. We write $\Pi^{\mathrm{reg}}(A)$ for the set of values admitting
at least one regular production and define
$$
c_n^{\mathrm{reg}}
:=
\max_{A\in\R^{n\times n}}
|\Pi^{\mathrm{reg}}(A)|.
$$
\end{definition}

We now define the principally simple matrices used below, following
Kielstra \cite[Definition~3.5.1]{Kielstra2023}.

\begin{definition}
\label{def:principally-simple}
A matrix $A$ is principally simple if

\begin{enumerate}[label=\textup{(\roman*)}]
\item every real eigenvalue of every principal submatrix is
algebraically simple;
\item two distinct principal submatrices have no common real
eigenvalue.
\end{enumerate}
\end{definition}

Principal simplicity makes every production regular. Indeed, if the
slack of a production on $J$ vanishes in a row $i\notin J$, extension
by a zero coordinate gives a common real eigenvalue of $A_J$ and
$A_{J\cup\{i\}}$. The same argument shows that a Pareto value cannot
be produced by two distinct supports. Thus the assigned support map is
unique for a principally simple matrix.

For a matrix $M$ and a nonempty principal support $J$, write
$$
\chi_J^M(t):=\det(t\Id_{|J|}-M_J).
$$

The next proposition shows that principally simple matrices form an
open dense set.

\begin{proposition}
\label{prop:principal-density}
Principally simple matrices form an open dense subset of
$\R^{n\times n}$.
\end{proposition}

\begin{proof}
Openness follows from continuity of polynomial roots. At a principally
simple matrix, the finitely many real principal eigenvalues are simple
and mutually separated, while every nonreal principal eigenvalue stays
a positive distance from the real axis under a sufficiently small
perturbation.

For each $J\in\Supp_n$, let
$$
\chi_J(t):=\det(t\Id_{|J|}-A_J).
$$
The discriminant of $\chi_J$ and the resultant of $\chi_J,\chi_K$ for
$J\neq K$ are polynomials in the entries of $A$. None of these
polynomials is identically zero. The assertion for a discriminant
follows from a diagonal matrix with distinct entries.

For a fixed pair $J\neq K$, assume after exchanging the sets that
$i\in J\setminus K$. Choose distinct diagonal entries on $J$. On $J$,
use $i$ as the centre of an arrowhead matrix and choose
$$
a_{i\ell}a_{\ell i}\neq0
\qquad
\text{for every }\ell\in J\cap K.
$$
At the diagonal entry indexed by $\ell\in J\cap K$, expansion of the
arrowhead determinant gives $\chi_J(a_{\ell\ell})\neq0$. Now take
$A_K$ diagonal. Its entries on $J\cap K$ are already fixed, and its
entries on $K\setminus J$ can be chosen outside the finite spectrum of
$A_J$. This gives
$$
\operatorname{spec}(A_J)\cap\operatorname{spec}(A_K)
=\varnothing,
$$
so the corresponding resultant is not the zero polynomial.

The product of these finitely many discriminants and resultants is
therefore a nonzero polynomial. Its nonvanishing set is Euclidean-open
and dense, and every matrix in it is principally simple.
\end{proof}

A regular production persists on the same support under a small
perturbation.

\begin{proposition}
\label{prop:regular-stability}
Let $\lambda_1,\ldots,\lambda_p$ be distinct Pareto values of $A$ with
regular productions on supports $J_1,\ldots,J_p$. Every sufficiently
small perturbation of $A$ has $p$ distinct nearby values regularly
produced on the same supports.
\end{proposition}

\begin{proof}
For each $\ell$, normalize the positive eigenvector $u_\ell$ by
$$
\one_{J_\ell}^{\top}u_\ell=1.
$$
Since $\lambda_\ell$ is an algebraically simple real eigenvalue of the
real matrix $A_{J_\ell}$, the system
$$
M_{J_\ell}u=\mu u,
\qquad
\one_{J_\ell}^{\top}u=1
$$
has, by the implicit function theorem, a unique local real solution
$$
M\longmapsto
\bigl(\lambda_\ell(M),u_\ell(M)\bigr)
$$
through $(\lambda_\ell,u_\ell)$. The derivative with respect to
$(\mu,u)$ is invertible because $\lambda_\ell$ is algebraically simple;
this can be checked by testing the linearized eigenvalue equation
against a left eigenvector. Hence the eigenvalue and the normalized
eigenvector depend continuously on $M$ near $A$. The normalization
fixes the sign of the eigenvector, and $u_\ell(M)>0$ for $M$
sufficiently close to $A$.

The exterior slack
$$
M_{J_\ell^{\mathrm c},J_\ell}u_\ell(M)
$$
also depends continuously on $M$ and remains strictly positive.
Finally, choose pairwise disjoint real neighborhoods of
$\lambda_1,\ldots,\lambda_p$ and reduce the matrix neighborhood once
more so that the eigenvalue branches remain distinct. This is the
usual simple-eigenvalue perturbation argument
\cite[Chapter~6]{HornJohnson2013}.
\end{proof}

Combining density and stability gives the principally simple
perturbation needed later.

\begin{corollary}
\label{cor:regular-principal-perturbation}
If $A$ has $p$ distinct regularly produced Pareto values, every
neighborhood of $A$ contains a principally simple matrix, with nonzero
off-diagonal entries and pairwise distinct diagonal entries, which has
at least $p$ distinct regular Pareto values.
\end{corollary}

\begin{proof}
The conditions on the entries are open and dense. Intersect them with
the dense set in Proposition~\ref{prop:principal-density} and with the
stability neighborhood from
Proposition~\ref{prop:regular-stability}.
\end{proof}

\subsection{The sharp lower example}

The following matrix, due to Victor Qi, attains the lower bound.

\begin{proposition}
\label{prop:lower-bound}
The matrix
$$
A_\star=
\begin{pmatrix}
100&106&-18&-81\\
92&158&-24&-101\\
2&44&37&-7\\
21&38&0&2
\end{pmatrix}
$$
has $23$ distinct regular Pareto eigenvalues. In particular,
$$
c_4^{\mathrm{reg}}\geq23
\qquad\text{and}\qquad
c_4\geq23.
$$
\end{proposition}
\begin{proof}
This matrix was communicated by Victor Qi and reported by Adly and
Seeger \cite[Section~2.2.2, p.~309, Table~2]{AdlySeeger2011}. Their
table gives numerical approximations to $23$ Pareto eigenpairs and
their dual vectors. The exact symbolic computation described in
Remark~\ref{rem:maple-certificate} places these $23$ candidates in
pairwise disjoint rational isolating intervals and verifies the strict
support criterion. It also verifies
$$
\gcd\bigl(\chi_J^{A_\star},(\chi_J^{A_\star})'\bigr)=1
\qquad
\text{for every }J\in\Supp_4.
$$
Thus each corresponding principal eigenvalue is algebraically simple,
each production is regular, and the $23$ Pareto eigenvalues are
distinct.
\end{proof}

The exact computation is described below.

\begin{remark}
\label{rem:maple-certificate}
The $23$ Pareto eigenvalues of $A_\star$ were verified by an exact symbolic computation performed with Maple 2021.2. The computation uses only integer and rational arithmetic. It verifies that the $15$ principal characteristic polynomials are square-free, that each of the $23$ rational intervals contains exactly one real root, that all associated active coordinates and exterior slacks are strictly positive, and that the intervals are pairwise disjoint. No floating-point approximation from \cite[Table~2]{AdlySeeger2011} is used. The Maple worksheet and the complete output are available from the author upon request.
\end{remark}
%%%%%%%%%%%
\section{Fixed point formulation and parity}
\label{sec:parity}

\subsection{The normalized fixed point map}

The standard simplex is
$$
\Delta:=\{x\in\R_+^n:\one^\top x=1\}.
$$
For a vector $y$, the notation $y_+$ means the componentwise positive
part. The norm $\|A\|_1$ is induced by the vector $1$-norm.

Adly and Rammal proved the positive-part equivalence
\cite[Lemma~3 and Remark~3]{AdlyRammal2013}
$$
0\leq x\perp(\mu x-Cx)\geq0
\quad\Longleftrightarrow\quad
(Cx)_+=\mu x,
\qquad
\mu>0.
$$
Here $x\perp y$ means $\langle x,y\rangle=0$.
We apply it to $C=\gamma\Id_n-A$ and $\mu=\gamma-\lambda$.

The normalized form of this equation gives the required fixed point
representation.

\begin{proposition}
\label{prop:fixed-point-representation}
Let $A\in\R^{n\times n}$ and choose $\gamma>\|A\|_1$. The map
$$
F_A:\Delta\longrightarrow\Delta,
\qquad
F_A(x):=
\frac{((\gamma\Id_n-A)x)_+}
{\one^\top((\gamma\Id_n-A)x)_+},
$$
is well defined. Its fixed points are exactly the normalized Pareto
eigenvectors of $A$. If $x\in\Fix F_A$, its Pareto eigenvalue is
$$
\lambda
=
\gamma-\one^\top((\gamma\Id_n-A)x)_+.
$$
\end{proposition}

\begin{proof}
Put $B=\gamma\Id_n-A$. The denominator cannot vanish. Otherwise
$Bx\leq0$, whereas
$$
\one^\top Bx
=
\gamma-\one^\top Ax
\geq
\gamma-\|Ax\|_1
\geq
\gamma-\|A\|_1
>0.
$$

Let $x\in\Fix F_A$, put $J=\supp x$, and set
$$
\rho:=\one^\top(Bx)_+.
$$
The fixed point equation gives
$$
B_Jx_J=\rho x_J>0,
\qquad
B_{J^{\mathrm c},J}x_J\leq0.
$$
Thus $J$ produces $\lambda=\gamma-\rho$.

Conversely, suppose that $J$ produces $\lambda$ through $u>0$. The
eigenvalue bound gives
$$
|\lambda|
\leq
\|A_J\|_1
\leq
\|A\|_1
<\gamma.
$$
Extend $u$ by zero and normalize it in $\Delta$. Since
$$
(\gamma-\lambda)x-Bx=Ax-\lambda x,
$$
the positive-part equivalence gives
$$
(Bx)_+=(\gamma-\lambda)x.
$$
Its normalization is the fixed point equation for $F_A$.
\end{proof}

\subsection{Fixed point index and parity}

We use the fixed point index on compact polyhedra. Let $X$ be a compact
polyhedron, let $G:X\to X$ be continuous, and let $U$ be relatively
open in $X$ with
$$
\Fix G\cap\partial_XU=\varnothing,
$$
where $\partial_XU$ is the boundary relative to $X$. The integer
$\ind(G,U)$ denotes the fixed point index of $G$ in $U$, relative to
$X$. If $x$ is an isolated fixed point, then $\ind(G,x)$ is
$\ind(G,U)$ for any sufficiently small isolating neighborhood $U$ of
$x$. Excision makes this definition independent of $U$. We use the
existence, additivity, homotopy invariance, and commutativity properties
from \cite[Section~12, pp.~305-326]{GranasDugundji2003}.

For a continuous self-map $G$ of a finite polyhedron $X$, its
Lefschetz number is
$$
L(G)
:=
\sum_{q\geq0}(-1)^q
\operatorname{tr}\bigl(
G_{*q}:H_q(X;\mathbb Q)\to H_q(X;\mathbb Q)
\bigr),
$$
where $H_q(X;\mathbb Q)$ is singular homology with rational
coefficients and $G_{*q}$ is the induced map.

We now prove that $|\Pi(A)|$ is odd when every Pareto value is produced
by exactly one support and this production is regular.

\begin{theorem}
\label{thm:parity}
Let $A\in\R^{n\times n}$. Assume that every value in $\Pi(A)$ is
produced by exactly one support and that this production is regular.
Then
$$
|\Pi(A)|\equiv1\pmod2.
$$
\end{theorem}

\begin{proof}
Choose $\gamma>\|A\|_1$, put $B=\gamma\Id_n-A$, and define
$$
F:=F_A.
$$
Proposition~\ref{prop:fixed-point-representation} identifies its fixed
points with the normalized Pareto eigenvectors.

The hypothesis gives one normalized fixed point for each Pareto value.
It also makes every exterior slack strict and every active principal
eigenvalue simple. Hence
$$
|\Fix F|=|\Pi(A)|.
$$

Fix $x\in\Fix F$ with support $J$. Choose a relatively open isolating
neighborhood $U$ of $x$ so small that the positive coordinates of
$By$ are exactly those in $J$ for $y\in U$, and that $F(U)$ is
contained in the face
$$
\Delta_J:=\{y\in\Delta:\supp y\subseteq J\}.
$$
Regard $F|_U$ as a map $G:U\to\Delta_J$, and let
$i_J:\Delta_J\to\Delta$ be the inclusion. Then
$F|_U=i_J\circ G$. On $U\cap\Delta_J$, the composition
$G\circ i_J$ is the reduced map
$$
F_J(y)
=
\frac{B_Jy}{\one_J^\top B_Jy},
\qquad
y\in U\cap\Delta_J.
$$
We use the local commutativity property in the following form. If
$X,Y$ are polyhedra, $W$ is open in $X$, $f:W\to Y$ and $g:Y\to X$
are continuous, and the relevant fixed point sets are compact, then
$$
\ind_X(g\circ f,W)
=
\ind_Y(f\circ g,g^{-1}(W)).
$$
Apply this formula with
$$
X=\Delta,
\qquad
Y=\Delta_J,
\qquad
W=U,
\qquad
f=G,
\qquad
g=i_J.
$$
Since $i_J^{-1}(U)=U\cap\Delta_J$, the commutativity property
\cite[Section~12, Theorem~6.2(VII), pp.~317-318]{GranasDugundji2003} gives the exact
local identity
$$
\ind_\Delta(F,U)
=
\ind_{\Delta_J}(F_J,U\cap\Delta_J).
$$
This reduction remains valid when $x$ lies on the boundary of the
ambient simplex.

The vector $x_J$ is an eigenvector of $B_J$ for the simple eigenvalue
$\rho=\gamma-\lambda$. On the tangent space
$$
T_x\Delta_J
=
\{h\in\R^J:\one_J^\top h=0\},
$$
the derivative is
$$
DF_J(x)h
=
\frac{B_Jh-x_J\one_J^\top B_Jh}{\rho}.
$$
It has no eigenvalue equal to $1$. Indeed, suppose that
$DF_J(x)h=h$ and set $\alpha=\one_J^\top B_Jh$. Then
$$
(B_J-\rho\Id_{|J|})h=x_J\alpha.
$$
Choose a left eigenvector $w$ for $\rho$ with $w^\top x_J=1$.
Multiplication by $w^\top$ gives $\alpha=0$. Simplicity then gives
$h\in\R x_J$, and the tangent condition gives $h=0$.

Every fixed point is nondegenerate on its face and
$$
\ind(F,x)\in\{-1,1\}.
$$
Here we used the local index formula
\cite[Section~12, pp.~326-329]{GranasDugundji2003}. The simplex is a
compact contractible polyhedron. Hence
$$
H_0(\Delta;\mathbb Q)\simeq\mathbb Q,
\qquad
H_q(\Delta;\mathbb Q)=0\quad(q\geq1),
$$
and $L(F)=1$. Since $\Fix F$ is finite, the Lefschetz-Hopf formula
\cite[Section~16, pp.~447-451]{GranasDugundji2003} gives
$$
\sum_{x\in\Fix F}\ind(F,x)=L(F)=1.
$$
Modulo $2$, every local index is equal to $1$. Hence
$$
|\Pi(A)|
=
|\Fix F|
\equiv1\pmod2.
$$
\end{proof}

Principal simplicity implies the hypotheses of the parity theorem.

\begin{corollary}
\label{cor:principally-simple-parity}
If $A\in\R^{n\times n}$ is principally simple, then
$$
|\Pi(A)|\equiv1\pmod2.
$$
\end{corollary}

\begin{proof}
Every Pareto value then has a unique producing support, its principal
eigenvalue is algebraically simple, and all exterior slacks are strict.
Theorem~\ref{thm:parity} applies.
\end{proof}

The local calculation remains valid when only one production is
regular.

\begin{lemma}
\label{lem:regular-local-index}
Let $A\in\R^{n\times n}$, choose $\gamma>\|A\|_1$, and define
$$
F_A(x):=
\frac{((\gamma\Id_n-A)x)_+}
{\one^\top((\gamma\Id_n-A)x)_+},
\qquad x\in\Delta.
$$
A fixed point arising from a regular production is isolated, is
nondegenerate on its active face, and has local index $1$ or $-1$.
\end{lemma}

\begin{proof}
Let $J$ be the support and let $\rho=\gamma-\lambda$. Since the
exterior slacks are strict, the positive coordinates of
$(\gamma\Id_n-A)y$ are exactly those in $J$ for $y$ near the fixed
point. Hence the local map takes its values in $\Delta_J$.
On that face the local map is
$$
y\longmapsto
\frac{(\gamma\Id_{|J|}-A_J)y}
{\one_J^\top(\gamma\Id_{|J|}-A_J)y}.
$$
The eigenvalue $\rho$ is algebraically simple. The derivative
calculation in the proof of Theorem~\ref{thm:parity}, applied on
$T_x\Delta_J$, shows that $1$ is not an eigenvalue of the derivative.
The fixed point is therefore isolated and nondegenerate on its face.
By the same commutativity argument as in
Theorem~\ref{thm:parity}, its local index in $\Delta$ equals the local
index of the reduced map on $\Delta_J$. It is therefore $1$ or $-1$.
\end{proof}

A matrix with at least $24$ regular Pareto values would therefore lead
to a principally simple matrix with exactly $25$ values.

\begin{corollary}
\label{cor:regular-to-25}
If some matrix of order $4$  has at least $24$ distinct regular Pareto
values, then there is a principally simple matrix with nonzero
off-diagonal entries and pairwise distinct diagonal entries that has
exactly $25$ Pareto eigenvalues.
\end{corollary}

\begin{proof}
Apply Corollary~\ref{cor:regular-principal-perturbation}. The resulting
principally simple matrix has at least $24$ Pareto values. Its spectrum
has odd cardinality by
Corollary~\ref{cor:principally-simple-parity}, and it has at most
$26$ values by the paragraph following Corollary~1 in
\cite[p.~8]{BaillonSeeger2021}, using
\cite[Lemma~5.1]{Seeger1999}. Its cardinality is therefore $25$.
\end{proof}

\section{Local restrictions in order four}
\label{sec:local-identities}

In the regular case, a counterexample would lead to a principally
simple matrix with exactly $25$ Pareto eigenvalues. We now give the
restrictions used to exclude this case. They concern pair supports,
principal submatrices of order $3$, and determinants associated with
the full matrix.

Let $A\in\R^{4\times4}$ and fix an assigned support map $\tau$. Put
$$
d_i:=a_{ii}.
$$
For a triple $T\subseteq\Ind_4$, put
$$
N_T(A,\tau)
:=
\sum_{\substack{J\in\Supp_4\\J\subseteq T}}m_J(A,\tau).
$$
We write $N_T$ when $A$ and $\tau$ are fixed. The bounds $c_2=3$ and
$c_3=9$, applied to the corresponding principal submatrices, give
$$
m_i+m_j+m_{ij}\leq3,
\qquad
N_T\leq9.
$$
Braces and commas are omitted in support subscripts.

A pair support $ij$ is rich when
$$
p_{ij}(A)=2.
$$
We denote by $G_2(A)$ the graph on $\Ind_4$ whose edges are the rich
pair supports. A triple $T$ is saturated when $p_T(A)=3$. This
intrinsic condition concerns production by $T$ itself and differs from
the assigned full-capacity condition $N_T=9$ on the whole
order-three face.

\subsection{Pair supports and order-three faces}

The quadratic eigenvalue equation fixes the orientation of a rich pair.

\begin{lemma}
\label{lem:pair-orientation}
If $ij$ is rich, then
$$
a_{ij}a_{ji}<0,
\qquad
\sgn(a_{ij})=\sgn(d_j-d_i).
$$
\end{lemma}

\begin{proof}
Write a positive eigenvector as $(1,t)^\top$, with $t>0$. Elimination
of the eigenvalue gives
$$
a_{ij}t^2+(d_i-d_j)t-a_{ji}=0.
$$
Two distinct positive roots have positive product and sum. These two
conditions give the stated signs.
\end{proof}

A producing singleton fixes every incident rich-pair orientation.

\begin{corollary}
\label{cor:singleton-orientation}
If the singleton $i$ produces and $ij$ is rich, then
$$
d_i>d_j.
$$
\end{corollary}

\begin{proof}
Singleton production gives $a_{ji}\geq0$, while richness gives
$a_{ji}\neq0$. Hence $a_{ji}>0$.
Lemma~\ref{lem:pair-orientation}, applied with the reversed orientation,
gives
$$
\sgn(a_{ji})=\sgn(d_i-d_j).
$$
\end{proof}

Two producing singletons cannot be joined by a rich pair.

\begin{corollary}
\label{cor:singleton-independent}
The producing singleton vertices form an independent set in the
graph $G_2(A)$.
\end{corollary}

\begin{proof}
If $i$ and $j$ both produce, then $a_{ij}\geq0$ and $a_{ji}\geq0$.
This is incompatible with Lemma~\ref{lem:pair-orientation}.
\end{proof}

Rich pairs cannot form a triangle.

\begin{lemma}
\label{lem:triangle-free}
The graph $G_2(A)$ is triangle-free.
\end{lemma}

\begin{proof}
Suppose that $12,13,23$ are rich and, after relabeling, that
$d_1<d_2<d_3$. Rich-pair orientation gives
$$
a_{31}<0,
\qquad
a_{32}<0.
$$
The exterior slack in row $3$ is then negative for every positive
eigenvector on the pair $12$, a contradiction.
\end{proof}

The general pair bound improves by one in order $4$ .

\begin{lemma}
\label{lem:S2-bound}
For every matrix of order $4$  and every assigned support map,
$$
S_2\leq9.
$$
\end{lemma}

\begin{proof}
For every pair,
$$
m_{ij}\leq p_{ij}(A)
\leq1+\indic_{\{ij\text{ is rich}\}}.
$$
If $e$ is the number of edges of $G_2(A)$, the triangle-free property
and the elementary four-vertex bound give
$$
S_2\leq6+e\leq10.
$$
Assume that $S_2=10$. All these inequalities are equalities:
$G_2(A)$ has four edges and is a cycle $C_4$, every rich pair has two
assigned values, and each of the two nonedges has one. Choose a vertex
$r$ with maximal diagonal entry. Its two incident rich edges force the
corresponding off-diagonal entries in row $r$ to be negative. The
nonedge opposite $r$ joins its two neighbors, so its exterior slack in
row $r$ is negative. This contradicts its production.
\end{proof}

We next state the restrictions on principal submatrices of order $3$.
They do not require regularity. We also list the full-capacity profiles
used later.
For $T=\{i,j,k\}$ with $i<j<k$, a local profile is listed in the order
$$
(m_i,m_j,m_k;\,m_{ij},m_{ik},m_{jk};\,m_T).
$$
The order-three restrictions now take the following form.

\begin{proposition}
\label{prop:order-three-restrictions}
Let $A\in\R^{4\times4}$, let $\tau$ be an assigned support map, and let
$T=\{i,j,k\}\subseteq\Ind_4$.

\begin{enumerate}[label=\textup{(\roman*)}]
\item If all three singleton supports in $T$ produce, then $m_T\leq1$.
\item If $i$ produces and $ij,ik$ are rich, then $m_{jk}=0$.
\item If $i,j$ produce and $ik,jk$ are rich, then $m_T\leq2$.
\end{enumerate}

If $N_T=9$, then, up to a permutation of $T$, the detailed profile on
$T$ is one of
$$
(1,0,0;1,2,2;3),
\qquad
(1,1,0;1,2,1;3),
\qquad
(1,1,0;1,2,2;2).
$$
\end{proposition}

\begin{proof}
The restrictions and the full-capacity classification in order $3$
are \cite[Lemmas~3.5-3.7 and Theorem~5.1]{Adly2026}. If $N_T=9$, the
nine globally assigned values remain distinct Pareto eigenvalues of
$A_T$. Since $c_3=9$, there is no additional local value, and the cited
classification applies.
\end{proof}

\subsection{Observability signs}

The following determinants relate signs on a saturated triple to signs
in the full matrix. We use only their signs.

Let $e_1,\ldots,e_4$ be the standard basis of $\R^4$.
For $r\in\Ind_4$, define
$$
D_r
:=
\det
\begin{pmatrix}
e_r^\top\\
e_r^\top A\\
e_r^\top A^2\\
e_r^\top A^3
\end{pmatrix}.
$$
For a triple $T$ and $i\in T$, define
$$
\delta_{i,T}
:=
\det
\begin{pmatrix}
e_i^\top\\
e_i^\top A_T\\
e_i^\top A_T^2
\end{pmatrix},
$$
where the coordinates in $T$ are kept in increasing global order.
Write $\operatorname{pos}_T(i)$ for the position of $i$ in this order.

The next two identities relate rich pairs and saturated triples to the
full support.

\begin{lemma}
\label{lem:determinant-identities}
Let $T=\Ind_4\setminus\{r\}$ be saturated, and let
$u_1,u_2,u_3>0$ be eigenvectors corresponding to the three values
produced by $T$.

If the exterior slacks of these productions are strict in row $r$, put
$$
s_\ell:=A_{r,T}u_\ell>0.
$$
Then, for every $i\in T$,
$$
D_r
=
(-1)^{r+1}\delta_{i,T}
\frac{\prod_{\ell=1}^3s_\ell}
{\prod_{\ell=1}^3(u_\ell)_i}.
$$

If $T\setminus\{i\}=\{j,k\}$, with $j<k$, is rich, let
$t_-<t_+$ be its two positive eigenvector ratios and put
$$
L_i(t):=a_{ij}+a_{ik}t.
$$
Then
$$
\delta_{i,T}
=
(-1)^{\operatorname{pos}_T(i)-1}
a_{jk}L_i(t_-)L_i(t_+).
$$
\end{lemma}

\begin{proof}
Assume first that the exterior slacks in row $r$ are strict.
Let $U=(u_1,u_2,u_3)$ and set
$$
V(\lambda_1,\lambda_2,\lambda_3)
:=
\det(\lambda_\ell^{\,k-1})_{k,\ell=1}^3.
$$
This is the Vandermonde determinant of the three eigenvalues of $A_T$.
Multiplication of the local observability matrix by $U$ gives
$$
\delta_{i,T}\det U
=
\left(\prod_{\ell=1}^3(u_\ell)_i\right)
V(\lambda_1,\lambda_2,\lambda_3).
$$
Embed the $u_\ell$ in $\R^4$ with zero coordinate in row $r$, and put
$M=(e_r,u_1,u_2,u_3)$. Since
$$
Au_\ell=\lambda_\ell u_\ell+s_\ell e_r,
$$
multiplication of the full observability matrix by $M$, followed by
elementary row operations, gives
$$
D_r\det M
=
\left(\prod_{\ell=1}^3s_\ell\right)
V(\lambda_1,\lambda_2,\lambda_3).
$$
Expansion along the first column gives
$$
\det M=(-1)^{r+1}\det U,
$$
which proves the first identity.

For the second identity, no condition on the exterior row $r$ is
needed. Move coordinate $i$ to the first position and keep $j<k$.
Expansion along the first row gives
$$
(-1)^{\operatorname{pos}_T(i)-1}\delta_{i,T}
=
a_{ij}^2a_{jk}
+a_{ij}a_{ik}(d_k-d_j)
-a_{ik}^2a_{kj}.
$$
The ratios $t_\pm$ satisfy
$$
a_{jk}t^2+(d_j-d_k)t-a_{kj}=0.
$$
Their sum and product turn the right-hand side into
$a_{jk}L_i(t_-)L_i(t_+)$. The two factors $L_i(t_\pm)$ are the
slacks in row $i$ of the productions by $jk$. They are strict. Indeed,
a zero slack would extend the corresponding pair eigenvector, with a
zero coordinate at $i$, to an eigenvector of $A_T$. Since saturation
provides all three distinct eigenvalues of $A_T$, that eigenvalue is
simple and its eigenspace is generated by a positive vector. This is a
contradiction.
\end{proof}

If the full support produces four values, all four observability
determinants have the same sign.

\begin{corollary}
\label{cor:common-D-sign}
If $p_{\Ind_4}(A)=4$, then $D_1,D_2,D_3,D_4$ are nonzero and have the
same sign.
\end{corollary}

\begin{proof}
Write $A=V\Lambda V^{-1}$, where the columns of $V$ are the four
positive eigenvectors produced by the full support. Then
$$
D_r
=
\frac{\det(\lambda_\ell^{\,k-1})_{k,\ell=1}^4}
{\det V}
\prod_{\ell=1}^4V_{r\ell}.
$$
The first two factors are independent of $r$, and the last factor is
positive.
\end{proof}

In a saturated triple, the diagonal entry at the common vertex of two
rich pairs lies between the other two diagonal entries.

\begin{lemma}
\label{lem:rich-path-median}
Let $T=\{i,j,k\}$ be saturated. If $ij$ and $jk$ are rich, then $d_j$
lies strictly between $d_i$ and $d_k$.
\end{lemma}

\begin{proof}
Relabel the indices as $1,2,3$, with rich pairs $12$ and $23$. The
determinant identity gives
$$
\sgn\delta_{1,T}
=
\sgn a_{23}
=
\sgn(d_3-d_2)
$$
and
$$
\sgn\delta_{3,T}
=
\sgn a_{12}
=
\sgn(d_2-d_1).
$$
The three local observability determinants have the same sign by their
Vandermonde expressions. Hence
$$
(d_3-d_2)(d_2-d_1)>0.
$$
\end{proof}

The singleton and order-three restrictions will be used in a combined
form.

\begin{lemma}
\label{lem:singleton-rich-rules}
Let $i,j,k$ be distinct indices.

\begin{enumerate}[label=\textup{(\roman*)}]
\item If $i$ and $j$ produce, then $ij$ is not rich.
\item If $i$ produces and $ij$ is rich, then $d_i>d_j$.
\item If $i$ produces and $ij,ik$ are rich, then $m_{jk}=0$.
\item If $i,j$ produce and $ik,jk$ are rich, then $m_{ijk}\leq2$.
\item If $i$ produces, the triple $ijk$ is saturated, and $ij,jk$ are
rich, then
$$
d_i>d_j>d_k.
$$
\end{enumerate}
\end{lemma}

\begin{proof}
The first two assertions are
Corollaries~\ref{cor:singleton-independent}
and~\ref{cor:singleton-orientation}. Assertions \textup{(iii)} and
\textup{(iv)} are Proposition~\ref{prop:order-three-restrictions}.
Finally, \textup{(ii)} gives $d_i>d_j$, and
Lemma~\ref{lem:rich-path-median} places $d_j$ between $d_i$ and $d_k$.
\end{proof}

The observability identities also compare complementary faces.

\begin{lemma}
\label{lem:D-gluing}
Let $T=\Ind_4\setminus\{r\}$ be saturated, assume that its exterior
slacks are strict in row $r$, and suppose that
$T\setminus\{i\}=\{j,k\}$, with $j<k$, is rich. Then
$$
\sgn D_r
=
(-1)^{r+\operatorname{pos}_T(i)}
\sgn(d_k-d_j).
$$
\end{lemma}

\begin{proof}
Combine the two identities in
Lemma~\ref{lem:determinant-identities}. All slack and eigenvector
factors are positive, and Lemma~\ref{lem:pair-orientation} gives
$$
\sgn a_{jk}=\sgn(d_k-d_j).
$$
\end{proof}

When the full support produces four values, a rich pair cannot belong
to two saturated triples.

\begin{corollary}
\label{cor:rich-edge-occupancy}
Assume that $p_{\Ind_4}(A)=4$. A rich pair is contained in at most one
saturated triple.
\end{corollary}

\begin{proof}
Let $jk$ be rich, with $j<k$, and let $i,r$ be the other two indices.
Suppose that
$$
T=\Ind_4\setminus\{r\},
\qquad
T'=\Ind_4\setminus\{i\}
$$
are saturated. Their exterior slacks in the complementary rows are
strict. Indeed, a zero slack would extend a produced triple
eigenvector, with one zero coordinate, to an eigenvector of $A$.
Since $p_{\Ind_4}(A)=4$, the full support produces all four distinct
eigenvalues of $A$; each eigenspace is one-dimensional and generated
by a positive vector.
This is a contradiction.
Lemma~\ref{lem:D-gluing} gives
$$
\sgn D_r
=
(-1)^{r+\operatorname{pos}_T(i)}
\sgn(d_k-d_j)
$$
and
$$
\sgn D_i
=
(-1)^{i+\operatorname{pos}_{T'}(r)}
\sgn(d_k-d_j).
$$
The positions satisfy
$$
\operatorname{pos}_T(i)
=
i-\indic_{\{r<i\}},
\qquad
\operatorname{pos}_{T'}(r)
=
r-\indic_{\{i<r\}}.
$$
The sum of the exponents is odd. Thus $D_i$ and $D_r$ have opposite
signs, contrary to Corollary~\ref{cor:common-D-sign}.
\end{proof}

We also need the following elementary trace inequality.

\begin{lemma}
\label{lem:trace}
Let $T=\{i,j,k\}$ be saturated. If
$$
a_{ij}>0,
\qquad
a_{ik}>0,
$$
then
$$
d_j+d_k>2d_i.
$$
\end{lemma}

\begin{proof}
For every positive eigenvector $u$ of $A_T$,
$$
(\lambda-d_i)u_i
=
a_{ij}u_j+a_{ik}u_k
>0.
$$
All three eigenvalues of $A_T$ are larger than $d_i$. Their sum is the
trace of $A_T$, which gives the inequality.
\end{proof}

\section{The regular Pareto capacity}
\label{sec:regular-capacity}

Corollary~\ref{cor:regular-to-25} shows that the regular upper bound
reduces to excluding a principally simple matrix with $25$ Pareto
values, nonzero off-diagonal entries, and pairwise distinct diagonal
entries. All its productions are regular.

The argument has a short combinatorial outline:
$$
\text{$25$ values}
\ \Longrightarrow\
\text{nine aggregate profiles}
\ \Longrightarrow\
\text{ten detailed profiles}
\ \Longrightarrow\
\text{a contradiction}.
$$

Fix such a matrix and its unique assigned support map, and put
$$
s:=S_1,\qquad q:=m_{1234}.
$$
Principal simplicity gives a unique producing support for every value.
Consequently,
$$
m_J=p_J(A)
\qquad
\text{for every }J\in\Supp_4.
$$
Let $\mathcal V$ be the set of producing singleton indices, and let
$\mathcal R$ and $\mathcal Z$ be the sets of pairs satisfying,
respectively, $p_{ij}(A)=2$ and $p_{ij}(A)=0$. Write
$$
r:=|\mathcal R|,
\qquad
z:=|\mathcal Z|.
$$
We also write
$$
\mathcal U:=\{T\subseteq\Ind_4:|T|=3,\ p_T(A)<3\}
$$
for the set of nonsaturated triples.
Every remaining pair produces one value, and therefore
$$
S_2=6+r-z.
$$
The graph $G_2(A)$ is triangle-free. Its possible types with at least
two edges on four vertices will be used below:
$$
\begin{array}{c|c}
r&\text{types}\\
\hline
2&P_3,\ 2K_2\\
3&P_4,\ K_{1,3}\\
4&C_4
\end{array}
$$
Here $P_k$, $C_k$, and $K_{1,k}$ denote, respectively, a path, a
cycle, and a star; $2K_2$ denotes two disjoint edges.
For $r=2$, the two edges are adjacent or disjoint. For $r=3$, a
triangle-free graph is a star if it has a vertex of degree three, and is
$P_4$ otherwise. For $r=4$, the only triangle-free type is $C_4$.

\subsection{Aggregate profiles}

We first show that at most two singleton supports produce.

\begin{lemma}
\label{lem:S1-at-most-two}
Every principally simple profile of total $25$ satisfies
$$
S_1\leq2.
$$
\end{lemma}

\begin{proof}
Suppose first that $s=4$. By the consequence of principal simplicity
stated after Definition~\ref{def:principally-simple}, the exterior
slack of every producing singleton is strict. Hence
$$
a_{ji}>0
\qquad
(i\neq j).
$$
After adding a sufficiently large scalar matrix, every principal
submatrix is positive. If a support $J$ produces a value through
$u>0$, then $u$ is a positive eigenvector of the shifted positive
matrix on $J$. By the Perron-Frobenius theorem
\cite[Chapter~8]{HornJohnson2013}, its eigenvalue is the spectral
radius. Thus every support produces at most one value and the total is
at most
$$
|\Supp_4|=15.
$$

Suppose now that $s=3$, and let $v$ be the only nonproducing singleton.
Among the three pairs contained in $\mathcal V$, each produces at most
one value. Let $r_v$ be the number of rich pairs joining $v$ to
$\mathcal V$. Then
$$
S_2\leq6+r_v.
$$
Start from $S_3\leq4\cdot3=12$. The triple formed by the three
producing singletons produces at most one value, and therefore gives a
deficit of at least two. Moreover, every unordered pair of rich edges
$vi,vj$ determines a distinct triple $vij$. By
Lemma~\ref{lem:singleton-rich-rules}\textup{(iv)}, this triple produces
at most two values. These $\binom{r_v}{2}$ triples are distinct from
the triple formed by the producing singletons. Hence
$$
S_3
\leq
12-2-\binom{r_v}{2}
=
10-\binom{r_v}{2}.
$$
Since $q\leq4$,
$$
25
\leq
3+(6+r_v)+\left(10-\binom{r_v}{2}\right)+4
=
23+r_v-\binom{r_v}{2}.
$$
For $r_v=0,1,2,3$, the last member is, respectively,
$23,24,24,23$. It cannot reach $25$, a contradiction.
\end{proof}

We next prove that the full support produces at least three values.

\begin{lemma}
\label{lem:q-not-two}
Every principally simple profile of total $25$ satisfies
$$
q\geq3.
$$
\end{lemma}

\begin{proof}
The bounds
$$
s\leq2,\qquad S_2\leq9,\qquad S_3\leq12
$$
already give $q\geq2$. Suppose that $q=2$. Then
$$
25
=
S_1+S_2+S_3+q
\leq
2+9+12+2
=
25,
$$
so equality holds in every term:
$$
(S_1,S_2,S_3,q)=(2,9,12,2).
$$
For a triple $T$, let $N_T$ be the total number of productions supported
in $T$. Double counting gives
$$
\sum_{|T|=3}N_T=3S_1+2S_2+S_3=36.
$$
Since $N_T\leq9$, all four faces have $N_T=9$.

Relabel the two producing singleton vertices as $1,2$. Since
$S_3=12$ and every triple produces at most three values,
$$
m_{123}=m_{124}=m_{134}=m_{234}=3.
$$
On the face $123$, there are two producing singletons and the
three-point support produces three values. Among the full-capacity
profiles in Proposition~\ref{prop:order-three-restrictions}, only
$$
(1,1,0;\,1,2,1;\,3)
$$
can occur, up to exchanging the two producing singleton vertices.
Consequently,
$$
m_{12}=1,
\qquad
|\mathcal R\cap\{13,23\}|=1.
$$
The same profile on $124$ gives
$$
|\mathcal R\cap\{14,24\}|=1.
$$

On the face $134$, there is one producing singleton and
$m_{134}=3$. The applicable profile is
$$
(1,0,0;\,1,2,2;\,3),
$$
up to exchanging $3$ and $4$. Hence $34$ is rich and exactly one of
$13,14$ is rich, so
$$
|\mathcal R\cap\{13,14\}|=1.
$$
The same profile on $234$ gives
$$
|\mathcal R\cap\{23,24\}|=1.
$$
For clarity, put
$$
x_{ij}:=\indic_{\{ij\in\mathcal R\}}.
$$
The four relations are
$$
x_{13}+x_{23}=1,
\qquad
x_{14}+x_{24}=1,
\qquad
x_{13}+x_{14}=1,
\qquad
x_{23}+x_{24}=1.
$$
Their two solutions are exchanged by the permutation
$3\leftrightarrow4$. Since $34$ is rich, we may assume
$$
\mathcal R=\{13,24,34\}.
$$
The saturated triple $134$, with singleton $1$ and rich path
$1$-$3$-$4$, gives
$$
d_1>d_3>d_4.
$$
The saturated triple $234$, with singleton $2$ and rich path
$2$-$4$-$3$, gives
$$
d_2>d_4>d_3.
$$
The inequalities $d_3>d_4$ and $d_4>d_3$ are incompatible.
\end{proof}

We now determine the possible aggregate profiles.

\begin{proposition}
\label{prop:nine-aggregate-profiles}
If a principally simple profile has total $25$, then its aggregate
profile belongs to the following list:
$$
\begin{array}{c|c}
q& (S_1,S_2,S_3,q)\\
\hline
3&
(1,9,12,3),\
(2,8,12,3),\
(2,9,11,3)
\\
4&
(0,9,12,4),\
(1,8,12,4),\
(1,9,11,4),\\
&
(2,7,12,4),\
(2,8,11,4),\
(2,9,10,4)
\end{array}
$$
\end{proposition}

\begin{proof}
Set
$$
\alpha:=9-S_2,\qquad
\beta:=12-S_3,\qquad
\gamma:=4-q.
$$
The bounds on $S_2$, $S_3$, and $q$ show that these three integers are
nonnegative.
The total-$25$ identity is equivalent to
$$
\alpha+\beta+\gamma=s.
$$
Lemmas~\ref{lem:S1-at-most-two} and~\ref{lem:q-not-two} give
$0\leq s\leq2$ and $\gamma\in\{0,1\}$. Listing the nonnegative solutions
of this equation gives exactly the nine rows above.
\end{proof}

To state the two detailed classifications with fixed labels, we first
record the ten candidate profiles. None of them has a zero pair. Every
pair outside $\mathcal R$ has multiplicity one, every triple outside
$\mathcal U$ has multiplicity three, and every triple in
$\mathcal U$ has multiplicity two. Hence the data in the table
determine every multiplicity $m_J$.

\begin{table}[ht]
\centering
\small
\caption{The ten candidate profiles of total $25$.}
\label{tab:ten-profiles}
\begin{tabular}{ccccc}
\toprule
Profile & $(S_1,S_2,S_3,q)$ & $\mathcal V$ & $\mathcal R$
& $\mathcal U$\\
\midrule
$\mathcal P_1$ & $(1,9,12,3)$ & $4$ & $13,23,24$ & $\varnothing$\\
$\mathcal P_2$ & $(2,8,12,3)$ & $3,4$ & $14,23$ & $\varnothing$\\
$\mathcal P_3$ & $(2,8,12,3)$ & $3,4$ & $12,24$ & $\varnothing$\\
$\mathcal P_4$ & $(2,9,11,3)$ & $3,4$ & $12,23,24$ & $234$\\
$\mathcal P_5$ & $(2,9,11,3)$ & $3,4$ & $12,14,23$ & $123$\\
\addlinespace
$\mathcal P_6$ & $(2,8,11,4)$ & $3,4$ & $23,24$ & $234$\\
$\mathcal P_7$ & $(2,8,11,4)$ & $3,4$ & $12,24$ & $124$\\
$\mathcal P_8$ & $(2,9,10,4)$ & $3,4$ & $12,23,24$ & $123,234$\\
$\mathcal P_9$ & $(2,9,10,4)$ & $3,4$ & $12,14,23$ & $123,124$\\
$\mathcal P_{10}$ & $(2,9,10,4)$ & $3,4$ & $12,14,23$ & $123,134$\\
\bottomrule
\end{tabular}
\end{table}

We classify the detailed profiles separately according to whether
$q=3$ or $q=4$.

\subsection{Detailed profiles with three full-support values}

The three aggregate profiles with $q=3$ have only five detailed
possibilities.

\begin{proposition}
\label{prop:q-three-classification}
Up to a simultaneous permutation of the indices, a principally simple
profile of total $25$ with $q=3$ is one of
$$
\mathcal P_1,\quad
\mathcal P_2,\quad
\mathcal P_3,\quad
\mathcal P_4,\quad
\mathcal P_5
$$
listed in Table~\ref{tab:ten-profiles}.
\end{proposition}

\begin{proof}
Consider first $(S_1,S_2,S_3,q)=(1,9,12,3)$. The identity
$S_2=6+r-z$ gives either
$(r,z)=(3,0)$ or $(r,z)=(4,1)$. Since every triple is saturated, a
singleton cannot be a vertex of degree at least two in $G_2(A)$:
its diagonal entry would be larger than those of two rich neighbors by
Lemma~\ref{lem:singleton-rich-rules}\textup{(ii)}, but would lie between
them by Lemma~\ref{lem:rich-path-median}. This excludes $C_4$ and the
centre of $K_{1,3}$. If the singleton is a leaf of $K_{1,3}$, denote the
centre by $c$ and the other two leaves by $j,k$. The two faces containing
the singleton and $c$ give $d_c>d_j,d_k$, whereas the face $cjk$ places
$d_c$ between $d_j$ and $d_k$. Thus the star is also impossible. Hence
$G_2(A)$ is $P_4$, with the singleton at an endpoint. This gives
$\mathcal P_1$.

Consider next $(2,8,12,3)$. Here
$$
(r,z)\in\{(2,0),(3,1),(4,2)\}.
$$
The cycle $C_4$ is excluded by the same median argument. If $r=3$, a
star forces the triple formed by its centre and the two singleton leaves
to be nonsaturated, by
Lemma~\ref{lem:singleton-rich-rules}\textup{(iv)}. For a path
$1$-$2$-$3$-$4$, the singleton pair is independent. It is therefore
either the endpoint pair $\{1,4\}$ or, up to reversal, the distance-two
pair $\{1,3\}$. The latter pair has the common rich neighbour $2$ and
again forces a nonsaturated triple. In the endpoint case,
$$
d_1>d_2>d_3
$$
follows from the saturated face $123$, while the saturated face $234$
gives $d_4>d_3>d_2$. This is impossible. Hence $r=2$ and $z=0$. If the
$G_2(A)$ is $2K_2$, the two singletons select one endpoint of each
edge, which gives $\mathcal P_2$. If it is $P_3$ with one isolated
vertex, its
middle vertex cannot be a singleton by
Lemma~\ref{lem:singleton-rich-rules}\textup{(iii)}, and its two endpoints
cannot both be singletons by
Lemma~\ref{lem:singleton-rich-rules}\textup{(iv)}. Thus one singleton is
an endpoint and the other is isolated. This gives $\mathcal P_3$.

It remains to consider $(2,9,11,3)$. There is exactly one nonsaturated
triple, and it produces two values. The identity $S_2=6+r-z$ gives
$(r,z)=(3,0)$ or $(4,1)$. In a $C_4$, the two singleton vertices are
opposite. Each of the two remaining vertices is their common rich
neighbour, so two different triples would be nonsaturated. This excludes
$C_4$. For $K_{1,3}$, the singletons are two leaves, and their face with
the centre is the unique nonsaturated triple. This is
$\mathcal P_4$. For $P_4$,
the singletons cannot be at distance two because the internal singleton
would force a zero pair while $z=0$. They are therefore the two
endpoints. The two three-vertex subpaths cannot both be saturated, by the
opposite inequalities obtained above. Exactly one of them is the
nonsaturated triple, and reversal of the path identifies the two
choices. This is $\mathcal P_5$.
\end{proof}

\subsection{Detailed profiles with four full-support values}

The next lemma excludes four aggregate profiles.

\begin{lemma}
\label{lem:q-four-first-removal}
No total-$25$ profile with $q=4$ has aggregate profile
$$
(0,9,12,4),\quad
(1,8,12,4),\quad
(2,7,12,4),\quad
(1,9,11,4).
$$
\end{lemma}

\begin{proof}
In each of the first three profiles every triple is saturated, while
$S_2=6+r-z$ gives $r\geq1$. Every rich edge is then contained in two
saturated triples, contrary to Corollary~\ref{cor:rich-edge-occupancy}.

For $(1,9,11,4)$, there is one nonsaturated triple. Every rich edge must
be contained in this triple, again by
Corollary~\ref{cor:rich-edge-occupancy}. A triangle-free graph on its three
vertices has at most two edges, whereas $S_2=6+r-z$ gives $r\geq3$.
\end{proof}

The two remaining aggregate profiles lead to five candidates.

\begin{proposition}
\label{prop:q-four-classification}
Up to a simultaneous permutation of the indices, a principally simple
profile of total $25$ with $q=4$ is one of
$$
\mathcal P_6,\quad
\mathcal P_7,\quad
\mathcal P_8,\quad
\mathcal P_9,\quad
\mathcal P_{10}
$$
listed in Table~\ref{tab:ten-profiles}.
\end{proposition}

\begin{proof}
Suppose first that the aggregate profile is $(2,8,11,4)$. There is one
nonsaturated triple $U$. Every rich edge lies in $U$, by
Corollary~\ref{cor:rich-edge-occupancy}. Thus $r\leq2$. The identity
$S_2=6+r-z$ gives
$r=2$ and $z=0$, and the two rich edges form a path inside $U$. Label
this path $1$-$2$-$3$, with $4\notin U$. The singleton pair cannot
contain two adjacent vertices. The choice $\{2,4\}$ is excluded because
the middle singleton would force $m_{13}=0$. The endpoint pair
$\{1,3\}$ gives $\mathcal P_6$, and an endpoint together with $4$
gives $\mathcal P_7$. These are the only choices, up to reversing the
path.

Suppose now that the aggregate profile is $(2,9,10,4)$. There are at
least two nonsaturated triples. Indeed, if there were only one, every
rich edge would lie in it. However, $S_2=6+r-z$ gives $r\geq3$,
whereas a triangle-free graph on the three vertices of this triple has
at most two edges. Hence
there are exactly two nonsaturated triples, and each produces two values.
Write them as
$$
U=\Ind_4\setminus\{a\},
\qquad
V=\Ind_4\setminus\{b\}.
$$
The only pair contained in neither $U$ nor $V$ is $ab$. Consequently
$$
ab\notin\mathcal R.
$$

The identity $S_2=6+r-z$ gives $(r,z)=(3,0)$ or $(4,1)$. If $r=4$, the graph $G_2(A)$ is
$C_4$, and the singleton vertices are opposite. Their two common rich
neighbours force the two corresponding triples to be $U$ and $V$. Each
remaining, saturated triple contains one singleton and its two rich
neighbours. The median rule and singleton orientation are incompatible
on either triple. Thus $r\neq4$.

Let $r=3$ and $z=0$. If $G_2(A)$ is a star, the singleton vertices
are two leaves. Their face with the centre is nonsaturated. The other
nonsaturated face cannot be the face formed by the three leaves, because
then the rich edge joining the centre to the remaining leaf would be
contained in neither nonsaturated face, contrary to
Corollary~\ref{cor:rich-edge-occupancy}. The two other choices are equivalent
by exchanging the singleton leaves, and give $\mathcal P_8$.

If $G_2(A)$ is a path, the singletons are its endpoints. Indeed, a
singleton at an internal vertex would force a zero pair. Label the path
$1$-$2$-$3$-$4$, with singletons $1,4$. At least one of $123$ and
$234$ is nonsaturated, since their simultaneous saturation gives both
$d_2>d_3$ and $d_3>d_2$. By reversing the path, take $123$ to be
nonsaturated. The second nonsaturated face can be $234$ or $134$.
The remaining choice $124$ is excluded by
Corollary~\ref{cor:rich-edge-occupancy}, since the two omitted vertices would
form the rich pair $34$. The first two choices give, respectively,
$\mathcal P_9$ and $\mathcal P_{10}$.
\end{proof}

\subsection{Elimination of the ten profiles}

The next lemma excludes $\mathcal P_6$ and $\mathcal P_{10}$ by
comparing the signs of the determinants $D_r$.

\begin{lemma}
\label{lem:remove-observability-profiles}
The profiles $\mathcal P_6$ and $\mathcal P_{10}$ are impossible.
\end{lemma}

\begin{proof}
Consider $\mathcal P_6$. The producing singleton $3$ gives
$a_{23}>0$, and
the producing singleton $4$ gives $a_{24}>0$. The face $123$ is
saturated, and the complementary rich pair $23$ gives
$$
\sgn D_4
=
(-1)^{4+\operatorname{pos}_{123}(1)}\sgn a_{23}
=-1.
$$
The face $124$ is saturated, and the complementary rich pair $24$ gives
$$
\sgn D_3
=
(-1)^{3+\operatorname{pos}_{124}(1)}\sgn a_{24}
=1.
$$
This contradicts the common sign of $D_1,D_2,D_3,D_4$.

Consider $\mathcal P_{10}$. The producing singletons give
$$
a_{14}>0,\qquad a_{23}>0.
$$
The saturated face $124$, with complementary rich pair $14$, gives
$$
\sgn D_3
=
(-1)^{3+\operatorname{pos}_{124}(2)}\sgn a_{14}
=-1.
$$
The saturated face $234$, with complementary rich pair $23$, gives
$$
\sgn D_1
=
(-1)^{1+\operatorname{pos}_{234}(4)}\sgn a_{23}
=1.
$$
This is the same contradiction.
\end{proof}

The Perron-Frobenius theorem excludes $\mathcal P_4$ and
$\mathcal P_8$.

\begin{lemma}
\label{lem:remove-positive-submatrix-profiles}
The profiles $\mathcal P_4$ and $\mathcal P_8$ are impossible.
\end{lemma}

\begin{proof}
Both profiles have producing singletons $3,4$, rich pairs $12,23,24$,
and a saturated triple $134$. Singleton production gives
$$
a_{13},a_{23},a_{43}>0,
\qquad
a_{14},a_{24},a_{34}>0.
$$
Richness of $23$ and $24$ then gives
$$
a_{32}<0,\qquad a_{42}<0.
$$
The exterior slacks of the producing pair $12$ in rows $3$ and $4$ are
strictly positive. Since the coefficients of the second coordinate are
negative, they force
$$
a_{31}>0,\qquad a_{41}>0.
$$
All six off-diagonal entries of $A_{134}$ are positive. After adding a
scalar matrix, $A_{134}$ is a positive matrix. Perron-Frobenius permits
only one eigenvalue with a positive eigenvector, whereas $m_{134}=3$.
This is impossible.
\end{proof}

Six profiles remain:
$$
\mathcal P_1,\quad
\mathcal P_2,\quad
\mathcal P_3,\quad
\mathcal P_5,\quad
\mathcal P_7,\quad
\mathcal P_9.
$$
They can be excluded directly.

\begin{proposition}
\label{prop:no-total-25}
No principally simple matrix of order $4$ , with nonzero off-diagonal
entries and pairwise distinct diagonal entries, has a support profile
of total $25$.
\end{proposition}

\begin{proof}
The preceding lemmas leave six cases.

\smallskip
\noindent\emph{Profile $\mathcal P_1$.}
The rich pairs $13,23$ in the saturated triple $123$ place $d_3$
strictly between $d_1$ and $d_2$. Applied in $234$ to singleton $4$
and the rich path $4$-$2$-$3$,
Lemma~\ref{lem:singleton-rich-rules}\textup{(v)} gives $d_2>d_3$.
Singleton $4$ and the rich pair $24$ give $d_4>d_2$. Hence
$$
d_1<d_3<d_2<d_4.
$$
Richness of $24$ gives $a_{42}<0$. The exterior slack of the rich pair
$23$ in row $4$ then gives $a_{43}>0$.

Suppose that $a_{41}\geq0$. For every positive eigenvector $u$ of
$A_{134}$,
$$
(\lambda-d_4)u_4
=
a_{41}u_1+a_{43}u_3
>0.
$$
All three produced eigenvalues of $A_{134}$ are larger than $d_4$.
Their sum is $d_1+d_3+d_4$, so
$$
d_1+d_3>2d_4,
$$
which is impossible because $d_1,d_3<d_4$. Thus $a_{41}<0$. The
producing pair $12$ now has exterior slack
$$
a_{41}u_1+a_{42}u_2<0
$$
in row $4$, a contradiction.

\smallskip
\noindent\emph{Profiles $\mathcal P_2,\mathcal P_5,\mathcal P_9$.}
Singletons $3,4$ and rich pairs $23,14$ give
$$
d_2<d_3,
\qquad
d_1<d_4,
\qquad
a_{32}<0,
\qquad
a_{41}<0.
$$
In all three profiles, pair $12$ produces and the triples $134$ and
$234$ are saturated. Its exterior slacks force
$$
a_{31}>0,
\qquad
a_{42}>0.
$$
Singleton production gives $a_{34}>0$ and $a_{43}>0$. The trace
inequalities in $134$ and $234$ give
$$
d_1+d_4>2d_3,
\qquad
d_2+d_3>2d_4.
$$
The first gives $d_3<d_4$, and the second gives $d_4<d_3$. This is
impossible.

\smallskip
\noindent\emph{Profiles $\mathcal P_3,\mathcal P_7$.}
For $\mathcal P_3$, apply
Lemma~\ref{lem:singleton-rich-rules}\textup{(v)} in the saturated
triple $124$ to singleton $4$ and the rich path
$4$-$2$-$1$. It gives $d_1<d_2$. For $\mathcal P_7$, the saturated
triple $234$ and the complementary rich pair $24$ give
$$
\sgn D_1=-\sgn(d_4-d_2)<0,
$$
while the saturated triple $123$ and the complementary rich pair $12$
give
$$
\sgn D_4=-\sgn(d_2-d_1).
$$
Since $q=p_{\Ind_4}(A)=4$, the common sign of the observability
determinants again gives $d_1<d_2$.

In both profiles, singleton $4$ and richness of $24$ now give
$$
d_1<d_2<d_4,
\qquad
a_{42}<0.
$$
Richness of $12$ gives $a_{12}>0$, and the row-$4$ slack of pair $12$
forces $a_{41}>0$. Singleton production gives
$$
a_{13},a_{14},a_{34},a_{43}>0.
$$
If $a_{31}\geq0$, then $a_{31}>0$ by the nonzero off-diagonal
assumption. After a scalar shift, $A_{134}$ is positive, contrary to
$m_{134}=3$ and the Perron-Frobenius theorem. Hence $a_{31}<0$, and
the row-$3$ slack of pair $12$ gives $a_{32}>0$.

The trace inequalities in $134$ and $234$ give
$$
d_1+d_3>2d_4,
\qquad
d_2+d_4>2d_3.
$$
Since $d_1,d_2<d_4$, the first gives $d_3>d_4$, whereas the second
gives $d_3<d_4$. This is impossible.
\end{proof}

The regular capacity now follows from the lower example and the
preceding exclusion.

\begin{theorem}
\label{thm:regular-capacity}
The regular Pareto capacity in order $4$  is
$$
c_4^{\mathrm{reg}}=23.
$$
\end{theorem}

\begin{proof}
Proposition~\ref{prop:lower-bound} gives the lower bound. If a matrix
had at least $24$ distinct regular Pareto values,
Corollary~\ref{cor:regular-to-25} would provide a matrix satisfying the
hypotheses of Proposition~\ref{prop:no-total-25} and having exactly
$25$ Pareto eigenvalues. This is impossible.
\end{proof}

\section{Reduction to principally simple matrices}
\label{sec:nonregular}

We now show that a matrix of order $4$  with at least $24$ Pareto
eigenvalues can be perturbed to a principally simple matrix without
decreasing their number. We first state the reduction theorem and then
prove the three lemmas used in its proof.

\begin{theorem}
\label{thm:reduction-to-principal}
Let $A\in\R^{4\times4}$ and suppose that
$$
p:=|\Pi(A)|\geq24.
$$
Every neighborhood of $A$ contains a principally simple matrix $C$
such that
$$
|\Pi(C)|\geq p.
$$
\end{theorem}

\subsection{Counting irregular productions}

We first show that at most two values are assigned to singleton
supports.

\begin{lemma}
\label{lem:large-spectrum-singletons}
Let $A\in\R^{4\times4}$ satisfy
$$
p:=|\Pi(A)|\geq24,
$$
and let $\tau$ be an assigned support map. Then
$$
S_1(A,\tau)\leq2.
$$
\end{lemma}

\begin{proof}
We omit $(A,\tau)$ from the assigned counts and put
$$
q:=m_{1234}.
$$
Thus
$$
p=S_1+S_2+S_3+q.
$$

We first prove that $S_1\leq2$. If $S_1=4$, all singleton supports
produce. Adding a scalar matrix makes every principal submatrix
nonnegative without changing producing supports. If a nonnegative
matrix has a positive eigenvector, the corresponding eigenvalue is its
spectral radius.
Hence every support produces at most one value and $p\leq15$.

Suppose that $S_1=3$. Let $V$ be the three singleton vertices, let $v$
be the remaining vertex, and let $r_v$ be the number of rich pairs
joining $v$ to $V$. No pair inside $V$ is rich. Since
$$
m_{ij}\leq p_{ij}(A)
\leq1+\indic_{\{ij\text{ is rich}\}},
$$
Proposition~\ref{prop:order-three-restrictions} gives
$$
S_2\leq6+r_v,
\qquad
S_3\leq10-\binom{r_v}{2}.
$$
Indeed, the triple $V$ produces at most one value, and each pair of
rich edges $vi,vj$ forces $m_{vij}\leq2$. Therefore
$$
p
\leq
3+(6+r_v)+\left(10-\binom{r_v}{2}\right)+4
=
23+r_v-\binom{r_v}{2}
\leq24.
$$
Equality requires $p=24$, $q=4$, $r_v\in\{1,2\}$, and equality in
every support bound used above. In particular,
$p_{\Ind_4}(A)=4$. If $r_v=1$, all three triples
containing $v$ have $m_T=3$, hence are saturated. The unique rich edge
belongs to two of them. This contradicts
Corollary~\ref{cor:rich-edge-occupancy}.

Let $r_v=2$. Relabel so that $V=\{1,2,3\}$, $v=4$, and the rich pairs
are $14$ and $24$. Equality makes $134$ and $234$ saturated, while
Corollary~\ref{cor:singleton-orientation} gives
$$
d_1>d_4,
\qquad
d_2>d_4.
$$
The exterior slacks of the two triples are strict: a zero slack would
extend a produced triple eigenvector, with one zero coordinate, to an
eigenvector of $A$. Since $p_{\Ind_4}(A)=4$, the matrix $A$ has four
distinct real eigenvalues, and each eigenspace is one-dimensional and
generated by a positive vector. Applying
Lemma~\ref{lem:D-gluing} to $134$ and $234$ gives
$$
\sgn D_2=\sgn(d_4-d_1)<0,
\qquad
\sgn D_1=-\sgn(d_4-d_2)>0.
$$
This contradicts Corollary~\ref{cor:common-D-sign}. Thus $S_1\leq2$.
\end{proof}

We next bound the number of irregular assigned productions by counting
algebraic root occurrences.

\begin{lemma}
\label{lem:irregular-occurrence-budget}
Let $A\in\R^{4\times4}$, let $\tau$ be an assigned support map, and put
$$
p:=|\Pi(A)|.
$$
Denote by $I_{\mathrm{irr}}(A,\tau)$ the number of assigned
productions whose principal eigenvalue is multiple or whose exterior
slack has a zero coordinate. Then
$$
p-S_1(A,\tau)+I_{\mathrm{irr}}(A,\tau)\leq28.
$$
If $p\geq24$ and $\tau$ selects a regular production whenever the
corresponding value admits one, then
$$
|\Pi(A)\setminus\Pi^{\mathrm{reg}}(A)|
=
I_{\mathrm{irr}}(A,\tau)
\leq6.
$$
\end{lemma}

\begin{proof}
We omit $(A,\tau)$ from the notation and write
$I_{\mathrm{irr}}:=I_{\mathrm{irr}}(A,\tau)$. Consider the roots,
counted with algebraic multiplicity, of all principal characteristic
polynomials with support size at least two. Each occurrence is
identified by its support and by its position among the copies of a
multiple root. There are
$$
6\cdot2+4\cdot3+4=28
$$
such occurrences. Assign one occurrence to each assigned nonsingleton
value. For each irregular nonsingleton value assigned to a support $J$,
assign one additional occurrence. If its root in $\chi_J^A$ is
multiple, use a second copy of this root. Otherwise, choose a zero
exterior slack in row $i$ and use the corresponding root of
$A_{J\cup\{i\}}$, obtained by extending the eigenvector by zero. For an
irregular singleton, use the pair root obtained from a zero exterior
slack, since a $1\times1$ eigenvalue is never multiple; no initial
occurrence was assigned to this value. For an irregular nonsingleton,
the additional occurrence differs from its initial occurrence.
Distinct values cannot use the same occurrence. Hence this assignment
is injective and
$$
p-S_1+I_{\mathrm{irr}}\leq28.
$$
If $p\geq24$, Lemma~\ref{lem:large-spectrum-singletons} gives
$S_1\leq2$. The convention on $\tau$ means that an assigned production
is irregular exactly when its Pareto value admits no regular
production. Therefore
$$
|\Pi(A)\setminus\Pi^{\mathrm{reg}}(A)|
=I_{\mathrm{irr}}
\leq28-p+S_1
\leq6.
$$
\end{proof}

\subsection{Simultaneous regularization}

At most three distinct Pareto eigenvalues can be regularized
simultaneously while keeping their producing vectors fixed.

\begin{lemma}
\label{lem:simultaneous-regularization}
Let $\lambda_1,\ldots,\lambda_k$ be distinct Pareto eigenvalues of
$A\in\R^{4\times4}$, where $1\leq k\leq3$. For each $\ell$, let
$J_\ell$ produce $\lambda_\ell$ through $u_\ell>0$. Then every
neighborhood of $A$ contains a matrix $B$ and pairwise distinct real
numbers $\beta_1,\ldots,\beta_k$ such that
$$
B_{J_\ell}u_\ell=\beta_\ell u_\ell,
\qquad
B_{J_\ell^{\mathrm c},J_\ell}u_\ell>0,
$$
and $\beta_\ell$ is an algebraically simple eigenvalue of
$B_{J_\ell}$ for every $\ell$. Thus all selected productions become
regular through the same vectors.
\end{lemma}

The proof of this perturbation lemma is given in
Appendix~\ref{app:simultaneous-regularization}.

\subsection{Index continuation}

The next lemma uses the fixed point index to perturb the matrix without
decreasing the number of Pareto eigenvalues.

\begin{lemma}
\label{lem:index-recovery}
Let $A\in\R^{4\times4}$ and suppose that
$$
|\Pi(A)\setminus\Pi^{\mathrm{reg}}(A)|\leq6.
$$
Then every neighborhood $\mathcal O$ of $A$ contains a principally
simple matrix $C$ such that
$$
|\Pi(C)|\geq|\Pi(A)|.
$$
\end{lemma}

\smallskip
\begin{proof}
Put $p:=|\Pi(A)|$ and fix a neighborhood $\mathcal O$ of $A$.
Choose $\gamma>\|A\|_1$ with a strict margin and, for every matrix $M$
near $A$, define
$$
F_M(x):=
\frac{((\gamma\Id_4-M)x)_+}
{\one^\top((\gamma\Id_4-M)x)_+},
\qquad x\in\Delta.
$$
For $\lambda\in\Pi(A)$, set
$$
K_\lambda
:=
\left\{
x\in\Fix F_A:
\gamma-\one^\top((\gamma\Id_4-A)x)_+=\lambda
\right\}.
$$
These compact sets are pairwise disjoint because a fixed point
determines its value through the displayed formula.
The spectrum is finite and
$$
\Fix F_A=\bigsqcup_{\lambda\in\Pi(A)}K_\lambda.
$$
Choose relatively open isolating neighborhoods $U_\lambda$ with
pairwise disjoint closures. Their boundaries relative to $\Delta$ are
compact and contain no fixed point of $F_A$. Uniform continuity of
$(M,x)\mapsto F_M(x)$ gives
a convex neighborhood
$\mathcal N\subset\mathcal O$ of $A$ such that no $F_M$,
$M\in\mathcal N$, has a fixed point on any
$\partial_\Delta U_\lambda$, and
$\gamma>\|M\|_1$ throughout $\mathcal N$.
Put $\nu_A(\lambda):=\ind(F_A,U_\lambda)$. Homotopy invariance
preserves these indices throughout $\mathcal N$.

Let $\mathcal G$ be the set of values $\lambda$ such that
$\nu_A(\lambda)=0$ and $\lambda$ admits a regular production. Let
$\mathcal H$ be the set of the remaining values with zero index. Put
$g:=|\mathcal G|$ and $h:=|\mathcal H|$. The hypothesis gives
$h\leq6$. For every value $\lambda\in\mathcal G$, choose one regular
production and shrink $\mathcal N$ to a smaller convex neighborhood of
$A$ so that this production persists and its normalized fixed point
remains in $U_\lambda$ throughout $\mathcal N$.

Choose
$$
k:=\left\lceil\frac h2\right\rceil\leq3
$$
and choose a set $\mathcal S\subseteq\mathcal H$ with
$|\mathcal S|=k$, taking $\mathcal S=\varnothing$ if $h=0$. Choose one
production for each $\lambda\in\mathcal S$. If $h=0$, put $B:=A$.
Otherwise, Lemma~\ref{lem:simultaneous-regularization}, applied inside
$\mathcal N$, gives a matrix $B\in\mathcal N$. For each
$\lambda\in\mathcal S$, the normalization of the selected producing
vector is then a regular fixed point of $F_B$ in $U_\lambda$. The
producing vector is unchanged, although its Pareto value may change
under the perturbation.

For each $\lambda\in\mathcal G\cup\mathcal S$, choose a relatively open
isolating neighborhood
$V_\lambda$ of the corresponding regular fixed point such that
$$
\overline{V_\lambda}\subset U_\lambda.
$$
Put
$$
W_\lambda
:=
U_\lambda\setminus\overline{V_\lambda}.
$$
This is a relatively open annulus in $\Delta$. By
Lemma~\ref{lem:regular-local-index},
$$
\ind(F_B,V_\lambda)
=:
\varepsilon_\lambda
\in\{-1,1\}.
$$
There is no fixed point on $\partial_\Delta V_\lambda$. Excision and
additivity therefore give
$$
\ind(F_B,U_\lambda)
=
\ind(F_B,V_\lambda)+\ind(F_B,W_\lambda).
$$
Since $\ind(F_B,U_\lambda)=0$, it follows that
$$
\ind(F_B,W_\lambda)
=
-\varepsilon_\lambda
\neq0.
$$

Let $\mathcal K$ be the union of all boundaries
$\partial_\Delta U_\mu$, $\mu\in\Pi(A)$, and all boundaries
$\partial_\Delta V_\lambda$ introduced above. This set is compact and
contains no fixed point of $F_B$. By uniform continuity, there is a
convex neighborhood $\mathcal B$ of $B$ such that
$$
\mathcal B\subset\mathcal N
$$
and no $F_M$, $M\in\mathcal B$, has a fixed point on $\mathcal K$.
By Proposition~\ref{prop:principal-density}, choose a principally
simple matrix $C\in\mathcal B$ sufficiently close to $B$. For
$$
M_t:=(1-t)B+tC,
\qquad
0\leq t\leq1,
$$
the whole segment lies in $\mathcal B$. Hence the homotopy
$t\mapsto F_{M_t}$ has no fixed point on any relevant boundary
$\partial_\Delta U_\mu$ or $\partial_\Delta V_\lambda$. Since
$$
\partial_\Delta W_\lambda
\subseteq
\partial_\Delta U_\lambda\cup\partial_\Delta V_\lambda,
$$
homotopy invariance preserves the indices on every $U_\mu$,
$V_\lambda$, and $W_\lambda$ from $B$ to $C$.

By the existence property of the fixed point index, each set
$U_\lambda$ with $\nu_A(\lambda)\neq0$ contains at least one fixed
point of $F_C$. For each $\lambda\in\mathcal G\cup\mathcal S$, both
$V_\lambda$ and $W_\lambda$ contain a fixed point of $F_C$.

There are $p-g-h$ sets $U_\lambda$ with nonzero index. Therefore
$$
|\Fix F_C|
\geq
(p-g-h)+2g+2k
=
p+g-h+2\left\lceil\frac h2\right\rceil
\geq p.
$$
For a principally simple matrix, fixed points correspond to distinct
regular Pareto values. Hence $|\Pi(C)|=|\Fix F_C|\geq p$. All
perturbations can be chosen inside $\mathcal O$.
\end{proof}

The preceding lemmas now prove the reduction theorem.

\begin{proof}
Fix a neighborhood $\mathcal O$ of $A$. Choose an assigned support map
$\tau$ which selects a regular production whenever the corresponding
Pareto value admits one. Lemma~\ref{lem:large-spectrum-singletons}
gives
$$
S_1(A,\tau)\leq2.
$$
Lemma~\ref{lem:irregular-occurrence-budget} then gives
$$
|\Pi(A)\setminus\Pi^{\mathrm{reg}}(A)|\leq6.
$$
Lemma~\ref{lem:index-recovery}, applied in $\mathcal O$, provides a
principally simple matrix $C\in\mathcal O$ such that
$$
|\Pi(C)|\geq|\Pi(A)|.
$$
\end{proof}

\section{Proof of the main theorem}
\label{sec:proof-main}

We now combine the sharp example, the regular upper bound, and the
reduction theorem.

\begin{proof}[Proof of Theorem~\ref{thm:main}]
Proposition~\ref{prop:lower-bound} gives $c_4\geq23$. Suppose that a
real matrix $A$ of order $4$  has at least $24$ distinct Pareto
eigenvalues. Theorem~\ref{thm:reduction-to-principal} gives a
principally simple matrix $C$ such that
$$
|\Pi(C)|\geq|\Pi(A)|.
$$
Every Pareto production of $C$ is regular. Hence
Theorem~\ref{thm:regular-capacity} gives
$$
24
\leq
|\Pi(A)|
\leq
|\Pi(C)|
=
|\Pi^{\mathrm{reg}}(C)|
\leq
c_4^{\mathrm{reg}}
=23,
$$
a contradiction. Therefore $c_4\leq23$. Combining both bounds yields
$$
c_4=23.
$$
This proves Theorem~\ref{thm:main}.
\end{proof}

We now compare our reduction theorem with Kielstra's earlier result.
\begin{remark}
Kielstra \cite[Corollary~6.1.7]{Kielstra2023} also obtained
$c_4=c_4^{\mathrm{reg}}$. His proof combines a computer-assisted
exhaustive enumeration of sign patterns of order $4$,
implemented by the code given in Appendix~A of his thesis, with
perturbation arguments. Our proof does not use this enumeration. The
regular upper bound follows from parity and restrictions on support
profiles. Theorem~\ref{thm:reduction-to-principal} uses a count of
algebraic root occurrences, simultaneous regularization, and
continuation of the fixed point index.
\end{remark}

\section{Conclusion and outlook}
\label{sec:conclusion}
In this paper, we proved that a real matrix of order four has at most $23$ distinct Pareto eigenvalues. We also gave an explicit matrix for which this bound is attained.

The proof is divided into two parts. First, we considered matrices for which every Pareto eigenvalue is produced by exactly one support and this production is regular. By using a fixed point formulation, we proved that the number of Pareto eigenvalues is odd. This result is valid in every order under the same assumptions. In order $4$, the support profiles of the principal submatrices, the rich-pair graph, and the signs of the observability determinants exclude the case of $25$ regular Pareto eigenvalues. Therefore,
$$
c_4^{\mathrm{reg}}=23.
$$

Second, we considered general real matrices of order four. A bound on the number of principal root occurrences, together with simultaneous regularization and continuation of the fixed point index, shows that a matrix with a nonregular Pareto spectrum can be replaced by a nearby regular matrix having at least the same number of Pareto eigenvalues. Consequently,
$$
c_4=c_4^{\mathrm{reg}}=23.
$$

It would be interesting to extend this approach to higher orders. The parity argument still applies under the regularity and uniqueness assumptions, but the number of principal root occurrences increases rapidly and the possible support profiles become more difficult to study. The reduction to regular spectra in higher orders is left for future work.

\appendix

\section{Simultaneous regularization}
\label{app:simultaneous-regularization}

We give the linear-algebraic construction used in
Lemma~\ref{lem:simultaneous-regularization}.

\begin{proof}[Proof of Lemma~\ref{lem:simultaneous-regularization}]
Extend each $u_\ell$ by zero to a vector $x_\ell\in\R^4$. No two of
these vectors are proportional, since their associated Pareto values
are distinct.

Suppose first that $x_1,\ldots,x_k$ are linearly independent. For
$J_\ell=\supp x_\ell$, choose $v_\ell$ equal to zero on $J_\ell$ and
strictly positive on $J_\ell^{\mathrm c}$. There is a linear map $H$
such that
$$
Hx_\ell=v_\ell
\qquad
(1\leq\ell\leq k).
$$
For $t>0$ small, all selected exterior slacks are strict in
$$
A_0:=A+tH,
$$
while the principal eigenvalue equations are unchanged.

Choose vectors $\ell_r\in\R^4$ such that
$$
\ell_r^\top x_s=\delta_{rs},
$$
where $\delta_{rs}$ is the Kronecker symbol, and put
$P_r=x_r\ell_r^\top$. On
$$
\R^{J_r}
=
\R x_r\oplus\bigl(\ker\ell_r^\top\cap\R^{J_r}\bigr),
$$
the matrices $(A_0)_{J_r}$ and $(P_r)_{J_r}$ have the forms
$$
\begin{pmatrix}
\lambda_r&*\\
0&C_r
\end{pmatrix},
\qquad
\begin{pmatrix}
1&0\\
0&0
\end{pmatrix}.
$$
For every sufficiently small coefficient outside a finite exceptional
set, adding this multiple of $P_r$ separates the eigenvalue associated
with $x_r$ from the quotient spectrum. Choose the coefficients
$\varepsilon_r$ successively. At each step, avoid these exceptional
values and take $\varepsilon_r$ small enough to preserve the simplicity
obtained at the preceding steps. Since
$$
P_rx_s=\delta_{rs}x_r,
$$
for the matrix
$$
B:=A_0+\sum_{r=1}^k\varepsilon_rP_r,
$$
put
$$
\beta_r:=\lambda_r+\varepsilon_r
\qquad
(1\leq r\leq k).
$$
Then $u_r$ remains an eigenvector of $B_{J_r}$, the principal
eigenvalue $\beta_r$ is simple, and the strict exterior slack is
unchanged. The coefficients may also be chosen so that the values
$\beta_1,\ldots,\beta_k$ remain pairwise distinct.

It remains to consider $k=3$ when the vectors are linearly dependent.
Since no two vectors are proportional, every coefficient in a
nontrivial dependence relation is nonzero. Because the vectors are
nonnegative, we may relabel them so that
$$
z=ax+by,
\qquad
a,b>0.
$$
Write $X=\supp x$, $Y=\supp y$, and
$Z=X\cup Y=\supp z$. The production equations give
$$
b(Ay)_i=a(\lambda_z-\lambda_x)x_i
\quad\text{for }i\in X\setminus Y,
$$
$$
a(Ax)_i=b(\lambda_z-\lambda_y)y_i
\quad\text{for }i\in Y\setminus X,
$$
and
$$
a(\lambda_x-\lambda_z)x_i
+
b(\lambda_y-\lambda_z)y_i=0
\quad\text{for }i\in X\cap Y.
$$
If $X\cap Y$, $X\setminus Y$, and $Y\setminus X$ are all nonempty,
the first two relations give
$$
\lambda_z\geq\lambda_x,
\qquad
\lambda_z\geq\lambda_y.
$$
Both inequalities are strict because the three values are distinct,
and this contradicts the third relation. Equal supports are also
impossible, since eigenvectors associated with three distinct
eigenvalues of the same principal matrix are linearly independent.
Thus the supports are disjoint or properly nested. The same relations
give, respectively,
$$
X\cap Y=\varnothing,
\qquad
\lambda_z>\max\{\lambda_x,\lambda_y\},
$$
or, after exchanging $x$ and $y$,
$$
Y\subsetneq X=Z,
\qquad
\lambda_x<\lambda_z<\lambda_y,
$$
and the restriction of $x$ to $Y$ belongs to $\R y$.

In the disjoint case, the first two relations make the slacks of $y$
on $X$ and of $x$ on $Y$ strict. In the nested case, the first
relation makes the slack of $y$ on $X\setminus Y$ strict. Thus, in
both configurations, only slacks outside $Z$ can vanish. Choose
$v_x,v_y$ equal to zero on $Z$ and positive on $Z^{\mathrm c}$, and
define
$$
Hx=v_x,
\qquad
Hy=v_y.
$$
Then $Hz=av_x+bv_y$. For $t>0$ small, the matrix
$$
A_0:=A+tH
$$
makes all three exterior slacks strict and preserves the selected
principal eigenvalue equations.

Set $W=\operatorname{span}\{x,y\}$. We construct a projection $P$ onto
$W$ which preserves the selected coordinate subspaces. If
$X\cap Y=\varnothing$, choose vectors
$\ell_x,\ell_y\in\R^4$, supported on $X,Y$, such that
$$
\ell_x^\top x=\ell_y^\top y=1,
$$
and put
$$
P=x\ell_x^\top+y\ell_y^\top.
$$
The relations above, which are unchanged for $A_0$ on $Z$, give
$$
(A_0)_Zx
=
\lambda_xx+\frac ba(\lambda_z-\lambda_y)y,
\qquad
(A_0)_Zy
=
\frac ab(\lambda_z-\lambda_x)x+\lambda_yy.
$$
Thus $W$ is invariant. Its two eigenvalues are $\lambda_z$ and
$\lambda_x+\lambda_y-\lambda_z$, which are distinct.

If $Y\subsetneq X$, choose a complement $L$ of $\R y$ in $\R^Y$.
Since $W\cap\R^Y=\R y$, extend $L$ to a complement $N_X$ of $W$ in
$\R^X$, and let $P$ be the projection onto $W$ along
$$
N_X\oplus\R^{X^{\mathrm c}}.
$$

For every selected support $J$, put $W_J=W\cap\R^J$. The active spaces
are
$$
W_X=\R x,\quad W_Y=\R y,\quad W_Z=W
$$
in the disjoint case, and
$$
W_Y=\R y,\quad W_X=W_Z=W
$$
in the nested case. Moreover,
$$
\R^J=W_J\oplus(\ker P\cap\R^J).
$$
The space $W_J$ is invariant under $(A_0)_J$, and every selected
eigenvalue is simple in the restriction to $W_J$. Relative to this
decomposition,
$$
(A_0)_J=
\begin{pmatrix}
R_J&*\\
0&Q_J
\end{pmatrix},
\qquad
P_J=
\begin{pmatrix}
\Id_{W_J}&0\\
0&0
\end{pmatrix}.
$$
A small scalar $s$, outside a finite set, separates every selected
eigenvalue of $R_J+s\Id_{W_J}$ from the spectrum of $Q_J$. Therefore
all selected principal eigenvalues of
$$
B:=A_0+sP
$$
are simple. Since $P$ fixes $x,y,z$, the three produced values are
$$
\beta_x=\lambda_x+s,
\qquad
\beta_y=\lambda_y+s,
\qquad
\beta_z=\lambda_z+s.
$$
They remain distinct and keep their strict exterior slacks. All
perturbations can be chosen in the prescribed neighborhood of $A$.
\end{proof}

\section*{Acknowledgements}
The author gratefully acknowledges support from the Math AmSud project N°51756TF (VIPS), ECOS Project C24E06 and the FMJH Gaspard Monge Program for Optimization and Data Science.

\end{document}